\documentclass[11pt, twoside, letterpaper]{amsart}

\usepackage{t1enc}
\usepackage{latexsym}
\usepackage{amssymb}
\usepackage{graphicx}
\usepackage{amsmath}
\usepackage{amsthm}
\usepackage{amsfonts}
\usepackage{mathtools}
\usepackage{comment}
\usepackage{mathrsfs}

\usepackage[all]{xy}
\usepackage[british]{babel}

\usepackage{color}
\usepackage[hypertexnames=false,
    pdftex,
	pdfpagemode=UseNone,
	breaklinks=true,
	extension=pdf,
	colorlinks=true,
	linkcolor=blue,
	citecolor=blue,
	urlcolor=blue,
]{hyperref}

\newcommand{\cyclic}{\mathop{\kern0.9ex{{+}\kern-2.10ex\raise-0.20
      ex\hbox{\Large\hbox{$\circlearrowright$}}}}\limits}

\newtheoremstyle{daniel}{3.0mm}{2mm}{\itshape}{}{\bfseries}{.}{1.5mm}{}
\theoremstyle{daniel}
\newtheorem{thm}{Theorem}[section]
\newtheorem{NotRem}[thm]{Notation and Remark}
\newtheorem{prop}[thm]{Proposition}
\newtheorem{Defi}[thm]{Definition}

\newtheorem{Exs}[thm]{Examples}
\newtheorem{Rems}[thm]{Remarks}

\newtheorem*{thm*}{Theorem}
\newtheorem*{cor*}{Corollary}
\newtheorem*{thm3.6}{Theorem 3.6}
\newtheorem*{thm4.3}{Theorem 4.3}

\newtheorem*{prop*}{Proposition}
\newtheorem*{Notation}{Notation}

\newtheorem{Th}{Theorem}[section]
\newtheorem{Lem}[thm]{Lemma}
\newtheorem{Cor}[thm]{Corollary}

\newtheorem{Def}[thm]{Definition}
\newtheorem{Not}[thm]{Notation}

\newtheorem{Rem}[thm]{Remark}

\newtheorem{Ex}[thm]{Example}

\newtheorem*{Setup}{Setup}

\newenvironment{rem}   {\begin{Rem}\em}{\end{Rem}}

\newcommand{\cA}{\mathcal{A}}

\def\cC{\mathcal C }

\def\cF{\mathcal F}

\def\cI{\mathcal I }

\def\cO{\mathcal O}

\def\cW{\mathcal W}

\def\gD{{\mathfrak D}}
\def\gE{{\mathfrak E}}
\def\gF{{\mathfrak F}}

\def\gW{{\mathfrak W}}
\def\gY{{\mathfrak Y}}

\newcommand{\C}{\mathbb{C}}

\newcommand{\N}{\mathbb{N}}
\renewcommand{\P}{\mathbb{P}}

\newcommand{\R}{\mathbb{R}}

\renewcommand{\H}{\mathbb{H}}
\renewcommand{\R}{\mathbb{R}}

\DeclareMathOperator{\Amp}{Amp}

\DeclareMathOperator{\Weak}{Weak}

\DeclareMathOperator{\Coh}{Coh}
\def\ch{\mbox{ch}}

\DeclareMathOperator{\cycle}{cycle}

\def\Spec{\mbox{Spec}}

\def\Todd{\mbox{Todd}}

\DeclareMathOperator{\Conv}{Conv}
\DeclareMathOperator{\Pseff}{Pseff}
\DeclareMathOperator{\Nef}{Nef}

\DeclareMathOperator{\Quot}{Quot}
\DeclareMathOperator{\BSS}{\mathbf{BSS}}

\usepackage{xcolor}

\numberwithin{equation}{section}

\begin{document}
\title[Semistability conditions]{Semistability conditions defined by ample classes}

\author{Damien M\'egy, Mihai Pavel, Matei Toma}
\address{ Universit\'e de Lorraine, CNRS, IECL, F-54000 Nancy, France
}
\email{damien.megy@univ-lorraine.fr}
\address{Institute of Mathematics of the Romanian Academy,
P.O. Box 1-764, 014700 Bucharest, Romania
}
\email{cpavel@imar.ro}
\address{ Universit\'e de Lorraine, CNRS, IECL, F-54000 Nancy, France
}
\email{Matei.Toma@univ-lorraine.fr}

\date{\today}

\begin{abstract}
We study a class of semistability conditions defined by a system of ample classes for coherent sheaves over a smooth projective variety. Under some necessary boundedness assumptions, we show the existence of a well-behaved  chamber structure for the variation of moduli spaces of sheaves with respect to the change of semistability.
\end{abstract}
\maketitle
\setcounter{tocdepth}{1}
\tableofcontents
\noindent


\section{Introduction}

When trying to construct moduli spaces of algebraic vector bundles of arbitrary rank on a smooth projective variety $X$ one is naturally led to restrict the set of isomorphism classes of vector bundles that one wishes to parameterize. For vector bundles over curves Mumford introduced a (semi)stability condition in \cite{MumfordGIT} and constructed the corresponding moduli space of stable vector bundles in \cite{Mumford1963ProjectiveInvariants}. Then Seshadri constructed a projective moduli space of semistable vector bundles over curves in \cite{Seshadri}. 
The developement of this theory for higher dimensional base varieties $X$, due to  Gieseker \cite{Gieseker}, Maruyama \cite{MaruyamaModuliII}, Simpson \cite{simpson1994moduli} and many others, used semistability conditions which depended on the choice of an ample divisor class on $X$. The question arised how the constructed moduli spaces depended on this choice. When $X$ is a surface it was shown (see \cite{Qin93,ellingsrud95variation,MatsukiWentworth}) that there is a rational locally finite polyhedral chamber structure on the ample cone $\Amp^1(X)$ of $X$ accounting for the variation of semistability (see our Definition \ref{def:chambers}) but this is not the case any longer in general when $\dim (X)\ge3$, cf. \cite{Qin93}. Even in cases when a chamber structure on $\Amp^1(X)$ with linear walls exists, walls do not necessarily contain rational classes, \cite{Schmitt}. It appeared therefore useful to think of different domains of parameters for semistability conditions which could allow well-behaved chamber structures and where a theory of variation of the respective moduli spaces could be performed similarly to the two-dimensional case, cf. \cite{GrebToma}, \cite{GrebRossToma-variation}, \cite{GrebRossToma-Semicontinuity}, \cite{GrebRossToma-MasterSpace}.

In a separate development  semistability conditions on additive categories other than  
just the category $\Coh(X)$ of coherent sheaves on a projective variety  
were introduced and studied.  
This was done 
by Rudakov in \cite{rudakov1997stability},  Joyce in \cite{JoyceIII}, \cite{joyce2021enumerative}, Bridgeland in \cite{Bridgeland}, Andr\'e in \cite{Andre2009} and Bayer in \cite{bayer2009polynomial},  to mention only a few.

More recently Alper, Halpern-Leistner and Heinloth established in  \cite{AlperHLH} necessary and sufficient conditions for an algebraic stack to admit a good moduli space in the sense of Alper \cite{Alper13}. They also defined an abstract notion of stability on abelian categories and applied their existence criterion to show that under suitable conditions the corresponding moduli stacks admit good moduli spaces  \cite[Theorem 7.27]{AlperHLH}. 

The purpose of  this paper is to propose and study a  parameter space for semistability conditions  on the category $\Coh(X)$ of coherent sheaves on a smooth projective variety $X$ which should have most of the good properties which we know from the classical cases and in addition  a rational locally finite chamber structure for the variation of semistability. 
More precisely for an $n$-dimensional smooth projective variety $X$ over an algebraically closed field the proposed parameter space is $\Amp^n(X)\times\ldots\times\Amp^0(X)$, where $\Amp^p(X)$ denotes the interior of the nef cone $\Nef^p(X)$ of $p$-codimensional numerical cycle classes on $X$, see Section \ref{sect:Preliminaries}. For an element $\alpha=(\alpha_n, \ldots,\alpha_0) \in \Amp^n(X)\times\ldots\times\Amp^0(X)$ we define an {\em $\alpha$-Hilbert polynomial} of a coherent sheaf $E$ on $X$ by \[
	P_\alpha(E,m) = \sum_{i=0}^n \frac{1}{i!}(\int_X \ch(E)\alpha_i \Todd_X)m^i.
\] 
and use it to further define {\em $\alpha$-semistability} of coherent sheaves (Definition \ref{def:alphaSs}). 
These semistability conditions are particular cases of the polynomial semistability conditions defined  by Bayer in \cite{bayer2009polynomial}, but are more general than the Gieseker-Maruyama conditions, the multi-Gieseker conditions from \cite{GrebRossToma-variation} or those coming from differential forms defined in \cite{TomaLimitareaII} in the complex analytic context, see Remarks \ref{rem:Bayer} and \ref{rem:degrees}. The novelty in this notion of $\alpha$-semistability is that it provides more flexibility from the perspective of variation and wall-crossing in higher dimensions, as one can vary independently the components of $\alpha$ inside their corresponding ample cones.

In Section \ref{sec:boundedness} we prove a generalization to the case of $\alpha$-Hilbert polynomials of Grothendieck's boundedness criterion from  \cite{GrothendieckHilbert}, Theorem \ref{th:boundednessCriterion} and Corollary \ref{cor:Grothendieck}. Consequences are the existence of relative Harder-Narasimhan filtrations, Theorem \ref{thm:HN}, the openness of semistability in flat families, Proposition \ref{prop:openness}, 
and Lemma \ref{lem:walls} which is the key to the finiteness of our chamber structures. We then describe some properties of the stack of $\alpha$-semistable sheaves of fixed numerical class on $X$ in Sections \ref{sec:stack} and \ref{sec:moduli}
and use the criterion of Alper, Halpern-Leistner and Heinloth \cite{AlperHLH}
to show that in characteristic zero and under a boundedness assumption this stack admits a proper good moduli space, Theorem \ref{thm:moduli}. Finally in Section \ref{sec:chambers}
we show again under a boundedness assumption the existence of a rational locally finite chamber structure on our parameter space $\Amp^n(X)\times\ldots\times\Amp^0(X)$ accounting for the change of $\alpha$-semistability on coherent sheaves of fixed numerical class on $X$, Propositions \ref{prop:chamber-slope}, \ref{prop:chamber-top}, \ref{prop:chamber-general}.

\subsection*{Acknowledgments}
We thank Arkadij Bojko for his remark which made us discover a mistake in a previous version of the paper.

This work was supported by the IRN ECO-Maths. The second author was also partly supported by the PNRR grant CF 44/14.11.2022 \textit{Cohomological Hall algebras
of smooth surfaces and applications}.


\section{Preliminaries}\label{sect:Preliminaries}

We fix a smooth projective variety $X$ of dimension $n\ge 2$ over an algebraically closed field $k$ and an ample line bundle $\cO_X(1)$ on $X$.  

\begin{Not}
    For a coherent sheaf $E$ on $X$ and a positive integer $s$ we shall denote by $N_s(E)$ the maximal coherent subsheaf of $E$ of dimension less than $s$. We shall further set 
    $E_{(s)}:=E/N_s(E)$.
    In the particular case when $s=d:=\dim E$, we shall also write  $T(E):=N_d(E)$ and $E_{pure}:=E_{(d)}$.
\end{Not}

\subsection{Degree functions.}
Let $N^p(X)_\R$ be the numerical group of real codimension $p$ cycles on $X$. We denote by $\Nef^p(X)$ the cone in $N^p(X)_\R$ dual to the pseudoeffective cone $\Pseff^{n-p}(X)$,   and by $\Amp^{p}(X)$ the interior of $\Nef^p(X)$. 
For $0\le p\le n$ the cones $\Nef^p(X)$ are full-dimensional (even for singular $X$ by \cite[Lemma 3.7]{FulgerLehmann2017cones}), so their interiors $\Amp^{p}(X)$ are non-empty.
We call the elements of $\Amp^{p}(X)$ {\em ample $p$-classes\footnote{This terminology appears already in \cite{Grothendieck-letter}.}}. Such an element $\alpha \in \Amp^p(X)$ defines a degree function
\[ 
	\deg_{\alpha} : K_0(X) \to \R
\]
by sending the class of a coherent sheaf $F$ on $X$ to $\int_X \ch(F)\alpha \Todd_X$.  For a closed subscheme $Y$ of $X$ we write simply $\deg_\alpha(Y) :=  \deg_\alpha([\cO_Y])$.

Choose a degree system $\alpha = (\alpha_d,\ldots,\alpha_r)$ with $0 \le r < d \le n$ given by ample classes $\alpha_i \in \Amp^{i}(X)$. For $E \in \Coh(X)$, we define its {\em $\alpha$-Hilbert polynomial} by
\[
	P_\alpha(E,m) = \sum_{i=r}^d \deg_{\alpha_i}(E)\frac{m^i}{i!}.
\]

Note that the degree of $P_\alpha(E)$ equals the dimension of the support of $E$ whenever $\dim(E) \ge r$ since $\deg_{\alpha_i}(E) = 0$ for $i > \dim(E)$. Also it is easy to see that $P_\alpha$ is additive in short exact sequences, and so it gives a group homomorphism $P_\alpha : K_0(X) \to \R[m]$. {For any $E \in \Coh(X)$ of dimension $e$ with $r \le e \le d$, we denote by 
\[
    p_\alpha(E) = \frac{P_\alpha(E)}{\deg_{\alpha_e}(E)}
\]
its \textit{reduced $\alpha$-Hilbert polynomial}.
}

Taking $\alpha = (h^n,\ldots,h,[X])$ with $h \in \Amp^1(X)$ integral we recover the classical notion of Hilbert polynomial, which we simply denote by $P_h$. In the sequel we will denote by $h$ the class of the fixed polarization  $\cO_X(1)$ on $X$.

\begin{rem}\label{rem:degrees} 
When $k=\C$ the degree functions $\deg_\alpha$ introduced above are {\em strong degree functions} as defined in   \cite[Definition 2.1]{TomaCriteria}. However they do not necessarily {\em come from differential forms} in the sense of \cite[Section 2.3]{TomaLimitareaII}. 

Indeed the fact that they satisfy condition $(2)$ of Definition 2.1 in \cite{TomaCriteria} is a consequence of our Lemma \ref{lem:Chow}. Requiring for our degree functions to come from differential forms would mean that 
$\Amp^p(X)\subset \Weak^p(X)$ for all $p$, where $\Weak^p(X)\subset N^p(X)$ is the cone represented by weakly positive $(p,p)$-forms on $X$. But this inclusion does not hold in general as shown in \cite[Theorem B]{DELV11}.
\end{rem}

\subsection{Semistability conditions}

Here we define a notion of semistability with respect to a degree system $\alpha = (\alpha_d,\ldots,\alpha_r)$ with $0 \le r < d \le n$ and $\alpha_i \in \Amp^i(X)$.

\begin{Def}\label{def:alphaSs}
A coherent sheaf $E$ of dimension $d$ on $X$ is {\em $\alpha$-semistable} (resp. {\em $\alpha$-stable)} if 
\begin{enumerate}
	\item $E$ is pure of dimension $d$ and
	\item for all subsheaves $F \subset E$ with $0 < \deg_{\alpha_d}(F) < \deg_{\alpha_d}(E)$ we have
	\[
		 p_\alpha(F,m) \leq p_\alpha(E,m) \text{ (resp. <) for }m \gg 0.
	\]
\end{enumerate}
\end{Def}

When $\alpha = (h^n,\ldots,h,[X])$, we recover the classical notion of Gieseker semistability. For $\alpha = (\alpha_d,\alpha_{d-1})$ we obtain a notion of slope-semistability for purely $d$-dimensional coherent sheaves on $X$ with respect to the slope $$\mu^{\alpha_{d-1}}_{\alpha_{d}} \coloneqq  \frac{\deg_{\alpha_{d-1}}}{\deg_{\alpha_{d}}}.$$

As in the case of Gieseker semistability one can prove the following:

\begin{prop}\label{prop:abelianCategory}
Let $\alpha = (\alpha_d,\ldots,\alpha_0)$ be a degree system of ample classes. Then the full subcategory of $\Coh(X)$ whose objects are the zero sheaf and all $\alpha$-semistable sheaves of dimension $d$ on $X$ with fixed reduced $\alpha$-Hilbert polynomial $p$ is abelian, noetherian, artinian and closed under extensions.
\end{prop} 

\begin{Ex}\label{ex:lineBundleSst}
Let $Y \subset X$ be an integral subscheme of dimension $d$ and let $L$ be a line bundle on $Y$. Then $L^{\oplus N}$ is $\alpha$-semistable for  any degree system $\alpha = (\alpha_d,\ldots,\alpha_r)$ and $N \in \N$. 
\end{Ex}

To see the above, complete $\alpha$ to some degree system $\overline{\alpha} = (\alpha_d,\ldots,\alpha_r,\ldots,\alpha_0)$ of ample classes. As $L$ is clearly $\overline{\alpha}$-semistable,
it follows by Proposition \ref{prop:abelianCategory} that $L^{\oplus N}$ is also $\overline{\alpha}$-semistable. Thus $L^{\oplus N}$ is in particular $\alpha$-semistable.
 
The above notion of semistability is more general than the multi-Gieseker semistability introduced in \cite{GrebRossToma-variation} as the following example shows.

\begin{Ex}
We present an example of a threefold $X$ where $\Conv(\{ \alpha \beta \mid \alpha, \beta \in \Amp^1(X) \}) \subsetneq \Amp^2(X)$ and where boundedness of semistability holds for three dimensional coherent sheaves with respect to any system of degree functions defined by ample classes $\alpha_j \in \Amp^j(X)$.
 \end{Ex}
To check this we use \cite[Example 3.11]{FulgerLehmann2017cones}.  Let $E = \cO \oplus \cO \oplus \cO(-1)$ on $\P^1$, let $X = \P(E)$, let $f$ be the class in $N^1(X)$ of a fibre of $X \to \P^1$ and let $\xi$ be the class of $\cO_{\P(E)}(1)$ in $N^1(X)$.  In {\em loc.\! cit.} the authors compute the nef cones and find $$\Nef^1(X) = \langle f, \xi + f \rangle,  \quad  \Nef^2(X) = \langle f \xi, \xi^2 + f\xi \rangle.$$ 
As in \cite[Section 6]{GrebRossToma-variation}, we set $$C^+(X) = \bigcup_{\beta \in \Amp^1(X)} \beta \cdot K^+_\beta, \text{ where } K^+_\beta := \{ \alpha \in N^1(X) \mid \alpha\beta^2 > 0,  \ \alpha^2 \beta > 0 \}.$$
Let $\beta = cf + \xi \in \Amp^1(X)$ (so $c > 1$). Using the Grothendieck relation we get $\xi^3 = -\xi^2 f = -1$, and a direct computation shows that
$$K^+_\beta = \{ af + b\xi \in N^1(X) \mid b > 0, \ 2a+(c-1)b > 0 \},$$ 
whence we obtain $C^+(X) = \Amp^2(X)$.

One easily sees that $\{ \alpha \beta \mid \alpha, \beta \in \Amp^1(X) \}$ is the interior of the convex cone $\langle f\xi, \xi^2 + 2f\xi \rangle$ and is thus strictly smaller than $\Amp^2(X)$.

Finally the fact that boundedness of slope-semistability holds for three dimensional coherent sheaves with respect to $\alpha \in C^+(X)$ is the content of \cite[Theorem 6.8]{GrebRossToma-variation}.

\begin{rem}\label{rem:Rudakov}
The notion of $\alpha$-semistability fits within the broader context of algebraic stability conditions introduced by Rudakov \cite{rudakov1997stability}. We consider $\alpha = (\alpha_d,\ldots,\alpha_0)$ a degree system of ample classes as above, and let $\cA = \Coh_d(X)$ be the abelian category of coherent sheaves on $X$ of dimension less or equal than $d$. We can define a preorder on $\cA$ as follows: For non-zero sheaves $E, F \in \cA$, denote
\begin{align*}
    E \asymp F :{}& \text{if } p_\alpha(E) = p_\alpha(F),\\
   E \prec F :{}&  \text{if } \deg P_\alpha(E) > \deg P_\alpha(F), \text{ or if }\deg P_\alpha(E) = \deg P_\alpha(F) \text{ and }\\
   {}& p_\alpha(E) < p_\alpha(F).
\end{align*}
Condition $p_\alpha(E) < p_\alpha(F)$ above is short for:  $p_\alpha(E,m) < p_\alpha(F,m)$ for $m \gg 0$. {We also write $E \preceq F$ when $E \prec F$ or $E \asymp F$.}

This closely resembles Rudakov's construction in  \cite[Section 2]{rudakov1997stability}, except that his polynomial functions are assumed to be integer-valued while the $\alpha$-Hilbert polynomials here are not so. One can check as in \cite[Section 2]{rudakov1997stability} that the above preorder satisfies the seesaw property and defines a stability structure on $\cA$ (in the sense of Rudakov). 
\end{rem}

\begin{rem}\label{rem:Bayer}
Here we see that $\alpha$-semistability with respect to $\alpha = (\alpha_n,\ldots,\alpha_0)$ is a particular example of polynomial semistability as defined by Bayer in \cite{bayer2009polynomial}.

We consider the abelian category $\cA = \Coh(X)$ of coherent sheaves on $X$. Choose complex numbers $\rho_0, \ldots,\rho_n$ in the upper half-plane
\[
    \H = \{ z \in \C \mid z \in \R_{> 0} \cdot e^{i\pi \phi(z)}, \phi(z) \in (0,1] \}
\]
such that $1 > \phi(\rho_0)> \ldots > \phi(\rho_n)$. For $E \in \Coh(X)$ we define its central charge
\[
	Z_E(m) = \sum_{i=0}^n \rho_i \deg_{\alpha_i}(E) m^i \in \C[m]
\]
with respect to the complete degree system $\alpha = (\alpha_n,\ldots,\alpha_0)$. This extends to give a group homomorphism $Z : K(X) \to \C[m]$.

For any non-zero $E \in \Coh(X)$, note that $Z_E(m) \in \H$ for $m \gg 0$. This induces a function germ,
\[
	\phi_E : (\R \cup \{ + \infty \}, + \infty) \to \R
\]
defined such that $Z_E(m) \in \R_{> 0} \cdot e^{i\pi \phi_E(m)}$ for $m \gg 0$. A non-zero sheaf $E \in \Coh(X)$ is called \textit{Z-semistable} if for all subsheaves $0 \neq F \subset E$, we have $\phi_F \preceq \phi_E$, i.e., $\phi_F(m) \leq \phi_E(m)$ for $m \gg 0$. The group homomorphism $Z$ is called a \textit{polynomial stability function}  in \cite[Definition 2.3.2]{bayer2009polynomial}. A straightforward computation shows  that $\phi_F \preceq \phi_E$ if and only if $p_\alpha(F) \le p_\alpha(E)$, therefore $Z$-semistability is equivalent to our notion of $\alpha$-semistability.
\end{rem}

For later use, here we observe that given any $\alpha$-semistability condition, this induces in particular a notion similar to the classical slope-semistability by using truncated reduced $\alpha$-Hilbert polynomials. If $p=m^d+c_{d-1}m^{d-1}+\ldots+c_rm^r$ is a polynomial with $r < d$ and $s$ is an integer such that $r \le s < d$, then we set
$$p_s=m^d+c_{d-1}m^{d-1}+\ldots+c_sm^s$$ to be 
the upper truncation of $p$ to degree $s$.

\begin{Def}
Let $r \le s < d$ be non-negative integers and $\alpha = (\alpha_d,\ldots,\alpha_r)$ a degree system of ample classes. We say a coherent sheaf $E$ of dimension $d$ on $X$ is $(\alpha,s)$-semistable if $E$ is semistable with respect to $(\alpha_d,\ldots,\alpha_s)$ (see Definition \ref{def:alphaSs}). 
\end{Def}

For $\alpha = (\alpha_d,\ldots,\alpha_r)$, it is clear that $(\alpha,r)$-semistability is the same as $\alpha$-semistability and that the $(\alpha,d-1)$-semistability coincides with the notion of slope-semistability defined with respect to the slope $\mu^{\alpha_{d-1}}_{\alpha_d}$ (see below Definition \ref{def:alphaSs}).

\subsection{Harder-Narasimhan property}\label{subsec:HN property}

In the following $\alpha = (\alpha_d,\ldots,\alpha_0)$ is a degree system of ample classes. We will see below in Proposition \ref{prop:HNproperty} that the notion of $\alpha$-semistability satisfies the \textit{Harder-Narasimhan (HN) property}, meaning that any pure sheaf $E$ of dimension $d$ on $X$ admits a unique filtration (called the Harder-Narasimhan filtration)
\[
    0 = E_0 \subset E_1 \subset \cdots \subset E_m = E
\]
such that each factor $E_i/E_{i-1}$ is $\alpha$-semistable of dimension $d$ with 
\[
    p_\alpha(E_1/E_0) > \cdots > p_\alpha(E_m/E_{m-1}).
\]

It is clear that showing the HN property for $\alpha$ will imply that semistability with respect to any truncation $(\alpha_d,\ldots,\alpha_r)$ of $\alpha$ also satisfies the HN property. 

Due to the discussion in Remark \ref{rem:Rudakov}, there is a preorder $\prec$ on $\Coh_d(X)$ induced by $\alpha$. To show the HN property is satisfied we will apply Rudakov's result \cite[Theorem 2]{rudakov1997stability}. As $\Coh_d(X)$ is noetherian, it remains to check that $(\Coh_d(X),\prec)$ is {\em weakly-artinian}, i.e. there are no infinite chains of (strict) inclusions 
    \[ 
        \ldots \subset E_{i+1} \subset E_i \subset \ldots \subset E_2 \subset E_1
    \]
   in $\Coh_d(X)$ with $E_i \preceq  E_{i+1}$ for all $i$.

The proof we give here is very similar to that of Proposition 2.6 in \cite{rudakov1997stability}, except that our $\alpha$-Hilbert polynomials are not integer-valued.

\begin{prop}\label{prop:HNproperty}
The notion of $\alpha$-semistability has the Harder-Narasimhan property.
\end{prop}
\begin{proof}
According to the above discussion, it is enough to check that $\Coh_d(X)$ is weakly-artinian with respect to the preorder $\prec$. 
Suppose, by contradiction, that there is an infinite chain of subsheaves
 \[ 
        \ldots \subsetneq  E_{i+1} \subsetneq E_i \subsetneq \ldots \subsetneq E_2 \subsetneq E_1
    \]
in $\Coh_d(X)$ with strict inclusion at every step and with $E_i \preceq E_{i+1}$ for all $i$. 
Without loss of generality, we may assume that all the $E_i$ have dimension $d$. It is clear that the sequence of $\alpha_d$-degrees of the sheaves $E_i$ is stationary, so we may as well suppose that it is constant. Under these assumptions we deduce
from the inequalities
  $$p_{\alpha}(E_1) \le p_{\alpha}(E_{2}) \le p_{\alpha}(E_{3}) \le \ldots$$
   a sequence of inequalities for the $\alpha$-Hilbert polynomials
  $$P_{\alpha}(E_1) \le P_{\alpha}(E_{2}) \le P_{\alpha}(E_{3}) \le \ldots.$$
Now the equality $P_{\alpha}(E_1) = P_{\alpha}(E_{2}) + P_{\alpha}(E_{1}/E_2) $ implies $P_{\alpha}(E_1) > P_{\alpha}(E_{2}) $, which is a contradiction.
\end{proof}


\section{A boundedness criterion}\label{sec:boundedness}

In this section we show versions of Grothendieck's boundedness criteria from  \cite[Section 2]{Grothendieck-theHilbertScheme} adapted to the case when the classical Hilbert polynomial is replaced by an $\alpha$-Hilbert polynomial as in Section \ref{sect:Preliminaries}.

Following \cite[Section 1.7]{HL} we will use a slightly simplified definition for boundedness of families of coherent sheaves. This definition is more restrictive than the one in \cite{Grothendieck-theHilbertScheme} but it will suit our purposes. 

\begin{Def}
    A set $\gF$ of isomorphism classes of coherent sheaves on $X$ is said to be {\em bounded}, if there exists a scheme $S$ of finite type over $k$ and a coherent sheaf $E$ on $X\times S$ such that $\gF$ is contained in the set $$ \{ [E|_{X\times \{ s\}}] \ ; \ s \ \text{ closed point of} \ S\}.$$
    Here we identified the fiber $X\times \{ s\}$ of the second projection $X\times S\to S$ with $X$ by the first projection. 
\end{Def}

\begin{rem}
\label{rem:boundedness}
It is easy to see that if $S$ is a scheme of finite type over $k$ and if $\gF$ is a set of isomorphism classes of coherent sheaves on the fibers of the projection $X\times S\to S$ which is bounded in the above sense, then there is a scheme $T$ of finite type over $S$ and a coherent sheaf $E$ on $X\times T$ such that $\gF$ seen now as a set of isomorphism classes of coherent sheaves on the fibers of  $X\times T\to T$ via $T\to S$ is contained in the set of isomorphism classes of coherent sheaves on the fibers of $(X\times T)/T$ defined by $E$. 
\end{rem}

\begin{Lem}\label{lem: elementary} 
Let $V$ be a real vector space of dimension $m$, let  $e_1,\ldots,e_m$ be a basis of $V$  and let  $v=\sum_{j=1}^m a_je_j\in V$ be such that $a_j>0$ for all $j=1,\ldots,m$. Then for any choice of real constants $c_1,\ldots, c_m,c$, the subset $K=K_{c_1,\ldots, c_m,c}:=\{ f\in V^* \ | \ f(e_j)\ge c_j \ \text{ for all } j, \  f(v)\le c\}$ of the dual space $V^*$ is compact. 
\end{Lem}
\begin{proof}
If we express an element $f$ of $K$ in terms of the dual basis as $f=\sum_{j=1}^m b_je_j^*$, we immediately obtain bounds
$$c_j\le b_j\le \frac{c-\sum_{i\ne j}^m a_ic_i}{a_j}$$
for each coefficient $b_j$ of $f$.
\end{proof}

\begin{Lem}\label{lem:Chow}
Let $v \in \Amp^d(X)$ and $\gY$ be a set of pure reduced closed subschemes of dimension $d$ of $X$.  If $\deg_{v}$ is upper bounded on $\gY$,  then $\gY$ is bounded.
\end{Lem}
\begin{proof}
Let $c > 0$ be an upper bound of $\deg_{v}$ on $\gY$. 
Choose a basis $e_1,\ldots,e_m$ of $N^d(X)_\R$ such that every $e_i$ is an ample $d$-class and $v=\sum_{j=1}^m a_je_j$ with $a_j>0$ for $j=1,\ldots,m$.  By Lemma \ref{lem: elementary} the set $$K := \{ \tau \in N^{n-d}(X)_\R \mid v\cdot \tau \leq c,  \ e_j \cdot \tau \geq 0 \text{ for all }j\}$$
is compact.  It follows that the degree function $\deg_{h^d}$ is bounded on $\gY$.  We conclude by Chow's Lemma, cf. \cite[Lemme 2.4]{Grothendieck-theHilbertScheme}, that $\gY$ is bounded.
\end{proof}

In the following we say that a set $\gF$ of isomorphism classes of coherent sheaves on $X$ is \textit{dominated} if there is some coherent sheaf $G$ on $X$ such that for any $[E] \in \gF$ there is a quotient $G \twoheadrightarrow E$.

\begin{Lem}\label{lem:boundednessI}
Let $\gF$ be a dominated set of isomorphism classes of coherent sheaves on $X$ of dimension at most $d$,  let $\alpha_d \in \Amp^d(X)$ and $\alpha_{d-1} \in \Amp^{d-1}(X)$.  If the degree function $\deg_{\alpha_d}$ is upper bounded on $\gF$,  then the degree function $\deg_{\alpha_{d-1}}$ is lower bounded on $\gF$.  If the latter function is upper bounded as well,  then the set $\gF_{(d)} := \{ [F_{(d)}] \mid [F] \in \gF  \}$ is bounded. 
\end{Lem}
\begin{proof}
We may assume that for some $m, N \in \N$ all $[F] \in \gF$ are dominated by $\cO_X(-m)^N$.  

We start by showing that $\deg_{\alpha_{d-1}}$ is lower bounded on $\gF$.  The assertion is clear if $\dim F<d$, so we consider only the case when $\dim F=d$ for all $[F] \in \gF$.  

Let $[F]$ be an element of $\gF$, let $Y=\cycle_d(F)$ be the $d$-dimensional support cycle of $F$ and let $Y=\sum_{i=1}^{e(F)} \nu_iY_i$ be its decomposition into prime cycles.   Since $\deg_{\alpha_{d}}(F)$ is bounded for $[F] \in \gF$,  it follows that in the notation above $e(F)$,  the $\nu_i$-s as well as the prime cycles $Y_i$ are bounded. 

Let $\nu(F)$ be the maximum multiplicity of $\cycle_d(F)$.  
We proceed by induction on $\nu(F)$.  First we treat the case $\nu(F)=1$.  Here we argue by induction on $e(F)$.  

Let $e(F) = 1$. 
We may suppose that $F$ is pure.   Indeed, if $F$ is not pure, then it fits in a short exact sequence
\[
	0 \to T(F) \to F \to F_{pure} \to 0,
\]
where $T(F)$ is the torsion of $F$. Hence
\[
	\deg_{\alpha_{d-1}}(F) = \deg_{\alpha_{d-1}}(F_{pure}) + \deg_{\alpha_{d-1}}(T(F)) \geq \deg_{\alpha_{d-1}}(F_{pure}).
\]

We have $$\cO_Y(-m)^N \twoheadrightarrow F$$ and we have seen  in Example \ref{ex:lineBundleSst} \color{black} that $\cO_Y(-m)^N$ is semistable with respect to the slope function defined by $\mu := \deg_{\alpha_{d-1}}/\deg_{\alpha_{d}}$.  As $\deg_{\alpha_{d}}(Y)$ is bounded,  by Lemma \ref{lem:Chow} we get that $\mu(\cO_Y(-m)^N)$ is bounded. Therefore $$\alpha_{d-1}(F) \geq \mu(\cO_Y(-m)^N)\alpha_d(F) ,$$
which combined with the fact that $\alpha_d(F)$ is bounded gives the desired lower bound for $\alpha_{d-1}(F)$.

Now let $e(F) > 1$.   Consider the following short exact sequence
$$0 \to \cI_{Y_{e(F)}}F \to F \to F_{Y_{e(F)}} \to 0.$$
As $F$ is dominated and the cycles $Y_{e(F)}$ are bounded,  the sheaves $F_{Y_{e(F)}}$ and $ \cI_{Y_{e(F)}}F$ are dominated by some $\cO_X(-m')^{N'}$, where $m'$ and $N'$ only depend on $\gF$.  Moreover,  note that $\deg_{\alpha_d}(\cI_{Y_{e(F)}}F),  \deg_{\alpha_d}(F_{Y_{e(F)}}) \leq \deg_{\alpha_d}(F)$.  Hence by the induction hypothesis $\deg_{\alpha_{d-1}}(\cI_{Y_{e(F)}}F)$ and $\deg_{\alpha_{d-1}}(F_{Y_{e(F)}})$ are lower bounded.  It follows that $\deg_{\alpha_{d-1}}(F)$ is also lower bounded, which proves the case $\nu(F)=1$.

Next we let $\nu(F) > 1$.  We have a short exact sequence
$$0 \to \cI_{Y_{red}}F \to F \to F_{Y_{red}} \to 0,$$
where $Y_{red} = \cup_i Y_i$.  Note that $\nu(\cI_{Y_{red}}F),  \nu(F_{Y_{red}}) < \nu(F)$, which allows us to check that $\deg_{\alpha_{d-1}}(F)$ is lower bounded in the same way as in the previous situation.

We now prove that under the supplementary assumption that the degree function $\deg_{\alpha_{d-1}}$ is upper bounded by some constant $c$ on $\gF$,  the set $\gF_{(d)}$ is bounded.  For this we use the fact we have just proved that for any $\beta \in \Amp^{d-1}(X)$,  the degree function $\deg_\beta$ is lower bounded on $\gF$.  Indeed, as in the proof of Lemma \ref{lem:Chow}, we choose a basis $\beta_1,\ldots,\beta_\rho$ of $N^{d-1}(X)_\R$ in $\Amp^{d-1}(X)$ such that $\alpha_{d-1}$ is a linear combination of $\beta_1,\ldots,\beta_\rho$ with positive coefficients.  Let $c_1,\ldots,c_\rho$ be lower bounds for the functions $\deg_{\beta_1},\ldots,\deg_{\beta_\rho}$ on $\gF$. Then Lemma \ref{lem: elementary} tells us that the set  $$K := \{ \tau \in N^{n-d+1}(X)_\R \mid \alpha_{d-1}\cdot \tau \leq c,  \ \beta_j \cdot \tau \geq c_j \text{ for all }j\}$$
is compact. Therefore the degree function $\deg_{h^{d-1}}$ is bounded on $K$, and thus also on $\gF$. By \cite[Lemme 2.5]{Grothendieck-theHilbertScheme} it follows that the set $\gF_{(d)}$ is bounded. 
\end{proof}

\begin{Th}\label{th:boundednessCriterion}
Let $\alpha_j \in \Amp^j(X)$ for $0 \leq j \leq n$,  let $s$ be an integer,  and let $\gF$ be a set of isomorphism classes of coherent sheaves on $X$ satisfying the following conditions:\\
(a) $\gF$ is dominated,\\
 $(b_s)$ the degree functions $\deg_{\alpha_j}$ are upper bounded on $\gF$ for $j \geq s-1$.
 
 Then $\gF_{(s)} := \{ [F_{(s)}] \mid [F] \in \gF  \}$ is bounded and for each ample $(s-2)$-class $\beta$ the degree function $\deg_\beta$ is lower bounded on $\gF$.
\end{Th}
\begin{proof}
The proof goes exactly as in \cite[Théorème 2.2, p258]{Grothendieck-theHilbertScheme}. 

We argue by induction on the maximum dimension $d$ of the elements of $\gF$.  For $d < 0$ the assertion is trivial,  so let now $d \geq 0$ and assume the statement is true for lower dimension. For the first part we may also assume that $d \ge s$. 

By Lemma \ref{lem:boundednessI} we obtain that $\gF_{(d)}$ is bounded, thus the degree functions $\deg_{\alpha_j}$ are bounded on $\gF_{(d)}$ for $j \geq s-1$. Now for each $[F] \in \gF$, we have a short exact sequence
\[
    0 \to N_d(F) \to F \to F_{(d)} \to 0.
\]

If $E$ is a coherent sheaf which dominates  $\gF$, it follows that the kernels of the compositions of morphisms $E\to F\to F_{(d)}$ form a bounded set, hence the set  $\mathfrak{N}_{d} \coloneqq \{[N_d(F)] \mid [F] \in \gF \} $ is also dominated. Furthermore, as the degree functions are additive in short exact sequences, by using assumption $(b_s)$ we deduce that $\deg_{\alpha_j}$ is upper bounded on $\mathfrak{N}_{d}$ for $j \geq s-1$. Therefore, by the induction hypothesis, we obtain that $(\mathfrak{N}_{d})_{(s)}$ is bounded. As the families $(\mathfrak{N}_{d})_{(s)}$ and $\gF_{(d)}$ are bounded, and every $F_{(s)}$ fits in an extension of the form
\[
    0 \to (N_d(F))_{(s)} \to F_{(s)} \to F_{(d)} \to 0,
\]
we conclude by \cite[Proposition 1.2 (ii)]{Grothendieck-theHilbertScheme} that $\gF_{(s)}$ is bounded.

For the second assertion, we will use the additivity of the degree functions in short exact sequences such as
\[
    0 \to N_s(F) \to F \to F_{(s)} \to 0,
\]
with $[F] \in \gF$. As before, $\mathfrak{N}_{s} \coloneqq \{[N_s(F)] \mid [F] \in \gF \} $ is dominated. Also, by $(b_s)$ and the fact that $\gF_{(s)}$ is bounded, we deduce that $\deg_{\alpha_{s-1}}$ is upper bounded on $\mathfrak{N}_{s}$. Thus it follows by Lemma \ref{lem:boundednessI} that $\deg_{\beta}$ is lower bounded on $\mathfrak{N}_{s}$ for any $\beta \in \Amp^{s-2}(X)$. From the additivity of $\deg_{\beta}$ in short exact sequences, we conclude that $\deg_{\beta}$ is lower bounded on $\gF$.
\end{proof}

\begin{Cor}\label{cor:Grothendieck}
Let $\alpha_j \in \Amp^j(X)$ for $0 \leq j \leq n$ and let $\gF$ be a set of isomorphism classes of coherent sheaves on $X$ satisfying the following conditions:\\
(a) $\gF$ is dominated,\\
 $(b)$ the degree functions $\deg_{\alpha_j}$ are upper bounded on $\gF$ for all $j$.
 
 Then $\gF$ is bounded.
\end{Cor}

A relative version of the above works by Remark \ref{rem:boundedness}.

Let $r, d$ be integers with $0 \le r < d \le \dim(X)$ and let $\alpha = (\alpha_d,\ldots,\alpha_r)$ be a degree system of ample classes $\alpha_i \in \Amp^i(X)$.

As in the classical case relative Harder-Narasimhan filtrations  exist for the notion of $\alpha$-semistability too. We recall the definition and the existence statement from \cite[Section 2.3]{HL}. The proof goes mutatis mutandis as in loc. cit., cf. \cite[Section 6.2]{TomaLimitareaII} for a proof in the analytic set-up. Note that our boundedness result, Theorem \ref{th:boundednessCriterion}, is needed in this argument.

\begin{Def}\label{def:relativeHN}
    Let $S$ be an integral scheme of finite type over $k$ and $E$ be a flat family of $d$-dimensional coherent sheaves on $X$ parameterized by $S$. Then a {\em relative Harder-Narasimhan filtration of $E$ for the $\alpha$-semistability} consists in a projective birational morphism $T\to S$ of integral schemes over $k$ and a filtration 
    $$0=HN_0(E)\subset  HN_1(E)\subset...\subset HN_l(E)=E_T$$
on $X\times T$ such that the factors $HN_i(E)/HN_{i-1}(E)$ are flat over $T$ for $1\le i\le l$ and which induces the absolute Harder-Narasimhan filtrations fibrewise over some dense Zariski open subset of $S$.
\end{Def}

\begin{Th}\label{thm:HN}
Let $S$, $E$ be as in Definition \ref{def:relativeHN}
 and assume morever that the general members of $E$ are pure. Then there exists a  relative Harder-Narasimhan filtration $(T\to S, HN_\bullet(E))$ of $E$ for the $\alpha$-semistability. Moreover this filtration has the following universal property:  if $f:T'\to S$ is a dominant morphism of integral $k$-schemes and if $F_\bullet$ is a filtration of $E_{T'}$ with flat factors, which coincides  fibrewise with the absolute Harder-Narasimhan filtration over general points $s\in S$, then $f$ factorizes over $T$ and $F_\bullet=HN_\bullet(E)_{T'}$.
\end{Th}


\section{The moduli stack of semistable sheaves}\label{sec:stack}

In this section we see that the notion of $\alpha$-semistability behaves well in flat families of sheaves by studying the properties of the corresponding moduli substack $\cC oh_X^{\alpha ss} \subset \cC oh_X$ of $\alpha$-semistable sheaves on $X$, where $\alpha=(\alpha_d,\ldots,\alpha_r)$ is a fixed class in $\Amp^d(X)\times\ldots\times\Amp(X)^r$ with $n\ge d>r\ge0$.

\subsection{Openness of semistability}

The following shows that $\cC oh_X^{\alpha ss}$ is an open substack in $\cC oh_X$.

\begin{prop}\label{prop:openness}
The property of being $\alpha$-semistable (geometrically $\alpha$-stable) is open in flat families.
\end{prop}
\begin{proof}
We only treat the case of $\alpha$-semistability, as the other one is similar. Let $S$ be a scheme of finite type over $k$, and let $\cF$ be an $S$-flat family of sheaves on $X$.  Since the property of being pure is open, we may assume that all fibers $\cF_s$ over $S$ are pure with fixed numerical class $\gamma \in K(X)_\text{num}$. Consider the set
\begin{align*}
     H = \{ [F'] \in K(X)_\text{num} :{}& \cF_s \to F' \text{ is a pure quotient of dimension $d$ for a} \\
      {}& \text{geometric point $s \in S$ such that $p_\alpha(F') < p_\alpha(\cF_s)$}\}.
\end{align*}

By Lemma \ref{lem:boundednessI} the set $H$ is finite. Consider the relative Quot scheme $\varphi : Q :=\Quot(\cF,H) \to S$ of quotients $[\cF_s \to F']$ with $[F'] \in H$. Since $Q$ is proper over $S$, its scheme-theoretic image $\varphi(Q)$ is a closed subset of $S$. We see that a fiber $\cF_s$ is $\alpha$-semistable if and only if $s$ is contained in the open set $S \setminus \varphi(Q)$.
\end{proof}

\subsection{Universal closedness}

We now state a "valuative criterion of properness" for $\alpha$-semistability following Langton, \cite{Langton}. Langton proves his result in the case of the classical slope-stability, i.e. for $d=n$, $r=d-1$ and $\alpha=(h^n,h^{n-1})$. A generalization to the case $\alpha=(h^d,...,h^r)$ is given by Huybrechts and Lehn in \cite[Theorem 2.B.1]{HL}. Applying their proof in our situation with $\alpha\in \Amp^d(X)\times\ldots\times\Amp(X)^r$ arbitrary encounters a difficulty related to the fact that in this case $\alpha$-degrees of coherent sheaves no longer range in a discrete set, since $\alpha$ is not necessarily rational. This difficulty has been circumvented in the proof of    \cite[Theorem 3.1]{TomaCriteria} and that proof applies to our present situation after performing the obvious changes. We refer the reader to it. 

\begin{Th}\label{thm:Langton}
  Let $R$ be a discrete valuation ring with residue field $k$ and quotient field $K$ and let $F$ be a coherent sheaf on $X_R$ giving a flat family of $d$-dimensional coherent sheaves on $X$ parameterized by $\Spec(R)$. If $F_K$ is $\alpha$-semistable on $X_K$ then there exists a coherent subsheaf $E$ of $F$ which coincides with $F$ over $\Spec(K)$ and whose fibre over $\Spec(k)$ is   $\alpha$-semistable.
\end{Th}

As a consequence of the above, one deduces that the moduli stack $\cC oh_X^{\alpha ss}$ satisfies the existence part of the valuative criterion of properness. However, as in the case of Gieseker semistability, this stack is not separated.

\subsection{Boundedness of semistability}\label{subsection:BSS}

The \textit{boundedness of semistability hypothesis}, ($\BSS$), with respect to a numerical class $\gamma \in K(X)_\text{num}$ of dimension $d > 0$ and a relatively compact subset $K \subset  \Amp^d(X) \times \Amp^{d-1}(X)$ is the following:
\vspace{0.5em}

$\BSS(K,\gamma)$:  The set of isomorphism classes of coherent sheaves $E$ of numerical class $\gamma$ on $X$ such that $E$ is slope-semistable with respect to some $(\alpha_d,\alpha_{d-1}) \in K$ is bounded. 
\vspace{0.5em}

The case $K = \{(\alpha_d,\alpha_{d-1})\}$ is particularly important to ensure the existence of \textit{finite-type} moduli stacks of pure sheaves, which will further provide one of the key ingredients in the construction of good moduli spaces of sheaves (in the sense of Alper \cite{Alper13}). It took the efforts of many mathematicians to solve this case in the context of classical slope-semistability (e.g.~\cite{atiyah57vector,takemoto1972stable,Gieseker,Maruyama81boundedness, simpson1994moduli, Langer}). 

For studying the variation of $\alpha$-semistability, it is crucial to allow $(\alpha_d,\alpha_{d-1})$ to vary in some compact $K$ and we have to ensure the boundedness statement $\BSS(K,\gamma)$ holds true, cf. Section \ref{sec:chambers}. In the torsion-free case, several such results were obtained by Matsuki and Wentworth \cite{MatsukiWentworth} over smooth surfaces and by Greb, Ross and Toma (see \cite[Section 6]{GrebRossToma-variation}) over higher-dimensional base spaces.

The boundedness problem was solved in full generality for one-dimensional pure sheaves (see \cite[Proposition 3.6]{greb2019semicontinuity} and \cite[Proposition 7.19]{joyce2021enumerative}). See \cite{PavelRossToma}
for the case of higher-dimensional pure sheaves.


\section{Existence of good moduli spaces of semistable sheaves}\label{sec:moduli}

In this section we apply \cite[Theorem 7.27]{AlperHLH} to show the existence of proper good moduli spaces for semistable sheaves on $X$ with respect to  complete systems of degree functions.  

Let $\gamma \in K(X)_\text{num}$ be a Grothendieck numerical class of dimension $d > 0$ which is represented by a coherent sheaf $F_0$.  We choose a degree system $\alpha = (\alpha_d,\ldots,\alpha_0)$ given by ample classes such that the boundedness statement $\BSS(\{(\alpha_d,\alpha_{d-1})\},\gamma)$ is satisfied.

Let $V$ be the real vector space of all polynomials in $\R[T]$ of degree at most $d$, endowed with the following total order relation: $P \leq Q$ if $P(t) \leq Q(t)$ for $t \gg 0$.  Consider the group morphism $p_\gamma : K(X) \to V$ defined by
$$ p_\gamma = \deg_{\alpha_d}(F_0)P_\alpha - P_\alpha(F_0)\deg_{\alpha_d},$$
where $P_\alpha$ is the $\alpha$-Hilbert polynomial defined in Section \ref{sect:Preliminaries}.  This defines a semistability notion as in \cite[Section 7.3]{AlperHLH} on the stack $\cC oh_X$ of flat families of coherent sheaves in $X$.  By definition a coherent sheaf $E$ on $X$ of class $\gamma$ is $p_\gamma$-semistable if for any subsheaf $F \subset E$ we have $p_\gamma(F) \leq 0$.  One checks that semistable sheaves in this sense are automatically pure and then easily sees that this notion of semistability coincides with that of $\alpha$-semistability.

The substack $\cC oh_X^{\gamma, \alpha ss}$ of $p_\gamma$-semistable points in $\cC oh_X$ is open by Proposition \ref{prop:openness} and quasi-compact by our assumption $\BSS(\{(\alpha_d,\alpha_{d-1})\},\gamma)$ .  We can now state the main result of this section:

\begin{thm}\label{thm:moduli}
In characteristic zero and under the boundedness assumption $\BSS(\{(\alpha_d,\alpha_{d-1})\},\gamma)$, the moduli stack  $\cC oh_X^{\gamma, \alpha ss}$ admits a proper good moduli space $M_X^{\gamma, \alpha ss}$.
\end{thm}
\begin{proof}
By \cite[Theorem 7.27]{AlperHLH} and the previous discussion, the stack $\cC oh_X^{\gamma, \alpha ss}$ admits a separated good moduli space which we call $M_X^{\gamma, \alpha ss}$.  Due to { Theorem \ref{thm:Langton} }, $M_X^{\gamma, \alpha ss}$ satisfies the existence part of the valuative criterion of properness, therefore it is proper. 
\end{proof}


\section{Chamber structure}\label{sec:chambers}

Let $\gamma \in K(X)_\text{num}$ be a Grothendieck numerical class of dimension $d \geq 1$ and let $r$ be a non-negative integer less than $d$. 
We will consider degree systems $ \alpha:=(\alpha_d,\ldots,\alpha_r)\in\Amp^d(X)\times\ldots\times\Amp^r(X)$ given by ample classes and we aim at introducing a locally finite  chamber structure on $\Amp^d(X)\times\ldots\times\Amp^r(X)$ describing the variation of $\alpha$-semistability of coherent sheaves with numerical class $\gamma$ with respect to $(\alpha_{d},\ldots,\alpha_r)$. 
In order to deal with the locally finite aspect of our desired chamber structure, we will restrict our attention to a compact subset $K$ of $\Amp^d(X)\times\ldots\times\Amp^r(X)$ and will show that under a boundedness of semistability assumption as in  Section \ref{subsection:BSS} the variation of $\alpha$-semistability on $K$ gives rise to a chamber structure involving only a finite number of walls. 
We give a formal definition. 
\begin{Def}\label{def:chambers}
For a fixed numerical class $\gamma \in K(X)_\text{num}$  of dimension $d \geq 1$ and  $0\le r<d$,  {\em a locally finite chamber structure for the variation of semistability on $\Amp^d(X)\times\ldots\times\Amp^r(X)$} is a way to associate to each compact subset $K$ of $\Amp^d(X)\times\ldots\times\Amp^r(X)$  a finite set $W(\gamma,K)$ of closed algebraic hypersurfaces in $\Amp^d(X)\times\ldots\times\Amp^r(X)$ called {\em walls} (or {\em $K$-walls}) with the property that for any connected component $C$ of $$(\Amp^d(X)\times\ldots\times\Amp^r(X)^{ss}_{(\gamma,K)}:=\Amp^d(X)\times\ldots\times\Amp^r(X)\setminus\bigcup_{w\in W(\gamma,K)}w,$$
any classes $\alpha, \ \alpha'\in C\cap K$ and any coherent sheaf $E$ in $\gamma$, $E$ is $\alpha$-semistable if and only if it is $\alpha'$-semistable. Such connected components $C$ will be called {\em chambers} (or {\em $K$-chambers}). When $K$ is fixed we will also speak of {\em $K$-chamber structure} when refering to $W(\gamma,K)$ and $(\Amp^d(X)\times\ldots\times\Amp^r(X)^{ss}_{(\gamma,K)}$. We will say that the chamber structure is {\em rational} if its walls are given by algebraic equations with rational coefficients. 
\end{Def}

We start by simplifying a bit the parameter space where we let $\alpha$ vary by noticing that a re-scaling by a positive real of any of the components of $\alpha$ does not change the $\alpha$-semistability. We thus may consider without loss of generality  sections $\Sigma_j$ of $\Amp^j(X)$ by (rational) affine linear hyperplanes $V_j$ in $N^j(X)$ such that $\Amp^j(X)$
become cones over $\Sigma_j$ and use $\Sigma:=\Sigma_d\times \ldots\times\Sigma_r$ as a parameter space for $\alpha$. We will also fix $K$ a compact subset of $\Sigma$. To further simplify the set-up, we will assume that the condition $\BSS(K', \gamma)$ from  Section \ref{subsection:BSS} is satisfied, where  $K'$ denotes the projection of $K$ to $\Sigma_d\times\Sigma_{d-1}$. 
The desired chamber structure on $\Sigma$ will be given by a finite set $W(\gamma,K)$ of walls in $\Sigma$ depending on the chosen compact $K$. The finiteness of $W(\gamma,K)$ will be a consequence of the hypothesis $\BSS(K', \gamma)$ and of the following boundedness statement. 
\begin{Lem}\label{lem:walls}
    Let $\gamma$, $\Sigma$ and $K$ be  as above and let $\gE$ be a bounded set of isomorphism classes of 
    coherent sheaves on $X$ in $\gamma$. We suppose $\gE$ to be  non-empty. Let further 
    $\gW(\gE,K)$ denote the set of isomorphism classes of pure $d$-dimensional proper quotients $F$ of some $E$ in $\gE$ which are weakly slope-destabilizing for $E$ with respect to  some $ \alpha\in K$, in the sense that $p_{\alpha,d-1}(F)\le p_{\alpha,d-1}(E)$. Then the set $\gW(\gE,K)$ is bounded.
\end{Lem}
\begin{proof}
 Note that for two compact subsets $K$ and $L$ of $\Sigma$ we have $\gW(\gE,K\cup L)=\gW(\gE,K)\cup\gW(\gE,L)$. Thus if $r<d-1$, we may replace $K$ by the larger compactum $K'\times K''$, where $K'$ denotes the projection of $K$ to $\Sigma_d\times\Sigma_{d-1}$ and $K''$ is the projection of $K$ to $\Sigma_{d-2}\times\ldots\times \Sigma_s$. (If $r=d-1$, we just take $K'=K$.) Clearly the set $\gW(\gE,K)$ will only depend on $K'$ and on $\gE$. Since any compact subset of $\Sigma_d\times\Sigma_{d-1}$ is covered by products $K_d\times K_{d-1}$ of compact polyhedrons $K_d$, $K_{d-1}$ in $\Sigma_d$ and $\Sigma_{d-1}$ respectively, we may suppose that our compact set $K'$ is of the form $K'=K_d\times K_{d-1}$ and $K_d$, $K_{d-1}$ are convex hulls of finite sets of points $S_d\subset\Sigma_d$ and $S_{d-1}\subset\Sigma_{d-1}$ respectively, $K_d=\Conv(S_d)$, $K_{d-1}=\Conv(S_{d-1})$. 

Since $\gE$ is bounded, the set $\gW(\gE,K)$ is dominated. To show its boundedness we need to bound the degrees of its elements. 

 For any $[F]\in \gW(\gE,K)$ the function  $$K\to\R, \ (\alpha_d,\alpha_{d-1})\mapsto \mu^{\alpha_{d-1}}_{\alpha_{d}}(F) $$ is of the form $(x,y)\mapsto\frac{f(x)}{g(y)}$, where $f$ and $g$ are affine linear, hence it attains its minimum on $S_d\times S_{d-1}$. 
If $E$ is a  coherent sheaf on $X$ in $\gamma$ we write $\mu^{\alpha_{d-1}}_{\alpha_{d}}(\gamma):=\mu^{\alpha_{d-1}}_{\alpha_{d}}(E)$ and set $M:= \max\{\mu^{\alpha_{d-1}}_{\alpha_{d}}(\gamma) \ ; \ (\alpha_d,\alpha_{d-1})\in K\}  $. 
Thus for each $[F]\in \gW(\gE,K)$ there exists a point $(\alpha_d,\alpha_{d-1})\in S_d\times S_{d-1}$ such that $\mu^{\alpha_{d-1}}_{\alpha_{d}}(F)\le M$. An application of Lemma \ref{lem:boundednessI} shows now that the set $\gW(\gE,K)$ is bounded. 
\end{proof}

Note that the above statement also holds for any bounded set $\gE$ of $d$-dimensional coherent sheaves on $X$. In that case we would deal with finitely many numerical classes $\gamma$ instead of just with one.

We will need also to deal with a set of "strongly" destabilizing quotients for a given set $\gE$  and a compactum $K$ as above. 

\begin{NotRem}\label{NotRem}
For $d>s\ge r$ we will denote by $\gD_s(\gE,K)$ the set of isomorphism classes of pure $d$-dimensional quotients $F$ of some $E$ in $\gE$ which are $(\alpha,s)$-destabilizing for $E$ with respect to  some $ \alpha\in K$, in the sense that $p_{\alpha,s}(F)< p_{\alpha,s}(E)$. If the set $\gE$ is bounded, then so is $\gD_s(\gE,K)$ too, since it is contained in $\gW(\gE,K)$.
\end{NotRem}

Before describing the general case we will consider two  cases of particular interest for the chamber structure on $\Sigma=\Sigma_d\times \ldots\times\Sigma_r$, the case when $r=d-1$ and the case when the compact subset $K\subset\Sigma$  is of the form $\{\alpha_d\}\times L$ for some compact subset $L$ of $\Sigma_{d-1}\times \ldots\times\Sigma_r$. Note that the latter always occurs if $d=n$. 

\subsection{The case of slope-semistability} This is the case when $r=d-1$, so $\alpha$ varies in $\Amp^d(X)\times\Amp^{d-1}(X)$.

Let $\gE(\gamma,K)$ be the set of isomorphism classes of pure coherent sheaves on $X$ in $\gamma$ which are $\alpha$-semistable with respect to some $\alpha$ in $K$. (We will suppose that $\gE(\gamma,K)$ is non-empty.)
Under the assumption 
$\BSS(K, \gamma)$ the set $\gE(\gamma,K)$ is bounded so we can  apply Lemma \ref{lem:walls} to it and get a bounded set  $\gW(\gE(\gamma,K),K)$ of weakly destabilizing quotients for $\gE(\gamma,K)$. 
We associate to each $[F]\in \gD_{d-1}(\gE(\gamma,K),K)$ a {\em $K$-wall} 
\begin{equation}\label{eq:wall}
w_F:=\{(\alpha_d,\alpha_{d-1})\in \Sigma \ ; \ \mu^{\alpha_{d-1}}_{\alpha_{d}}(F)=\mu^{\alpha_{d-1}}_{\alpha_{d}}(\gamma)\}.
\end{equation}
Since $\gD_{d-1}(\gE(\gamma,K),K)$ is bounded, the set
\begin{equation}\label{eq:wallset}
W(\gamma,K):= \{w_F \ ; \ [F]\in\gD_{d-1}(\gE(\gamma,K),K)\}
\end{equation}
of $K$-walls is finite. Thus we get the following
\begin{prop}\label{prop:chamber-slope} 
For $r=d-1$ and $K$ a compact subset of  $\Sigma_d\times\Sigma_{d-1}$ satisfying the condition $\BSS(K, \gamma)$, the set $W(\gamma,K)$ defined in \refeq{eq:wall} and \refeq{eq:wallset} above  builds a rational wall system for a finite $K$-chamber structure for the variation of semistability on $\Sigma_{d}\times\Sigma_{d-1}$.
\end{prop}

It is easily seen that by the definition of $\gD_{d-1}(\gE(\gamma,K),K)$ walls of type $w_F$ for $[F]\in \gD_{d-1}(\gE(\gamma,K),K)$ are non-empty proper closed subsets of $\Sigma$.
If both $\Sigma_d$ and $\Sigma_{d-1}$ are $0$-dimensional, then $\Sigma$ is a point, and there is no place for a chamber structure; the set of walls is empty in this case. 
If only one of $\Sigma_d$, $\Sigma_{d-1}$ is $0$-dimensional, then walls may exist and in this case they are intersections of linear hyperplanes with $\Sigma$. 
Note that this situation always occurs if $d=n$ or $d=1$. 
Finally if both $\Sigma_d$ and $\Sigma_{d-1}$ are positive dimensional, walls are given by (affine rational) bilinear equations in the variables $(x,y)\in V_d\times V_{d-1}$. In particular, also in this case walls are algebraic hypersurfaces as required.

\subsection{The case of fixed top degree} The following Proposition says that in this case the corresponding  chamber structure has a particularly well organized prismatic aspect. 
\begin{prop}\label{prop:chamber-top}
Let $\alpha_d\in \Sigma_d$ be a fixed ample $d$-class in $\Sigma_d$ and let $K$ be a subset of 
$\Amp^{d}(X)\times\ldots\times\Amp^r(X)$ 
of the form $K=\{\alpha_d\}\times L$, where $L$ is some compact subset of $\Sigma_{d-1}\times\ldots\times\Sigma_r$ such that the condition $\BSS(K', \gamma)$ is satisfied. Then for each $s$ with $d>s\ge r$ there exists a finite set $\cW_s$ of rational affine hyperplanes in $\Sigma_s$ such that 
$$W(\gamma,K):=\bigcup_{s=r}^{d-1} W(\gamma,K)_s$$ builds the wall system of a finite $K$-chamber structure for the variation of semistability on $\Sigma_{d}\times\ldots\times\Sigma_r$, where
$$W(\gamma,K)_s:=\{\Sigma_d\times\ldots\times\Sigma_{s+1}\times \omega\times\Sigma_{s-1}\times \ldots\times\Sigma_{r} \ ; \ \omega\in\cW_s\}.$$
\end{prop}

\begin{proof}
    For $K=\{\alpha_d\}\times L$ compact subset of $\Sigma_{d}\times\ldots\times\Sigma_r$,  we consider as before the set  $\gE(\gamma,K)$ of isomorphism classes of pure coherent sheaves on $X$ in $\gamma$ which are $\alpha$-semistable with respect to some $\alpha$ in $K$ and the set $\gW(\gE(\gamma,K),K)$ of weakly destabilizing quotients for $\gE(\gamma,K)$. Under the assumption $\BSS(K', \gamma)$ this set  is bounded by Lemma \ref{lem:walls}. Taking up the Notation and Remark \ref{NotRem} we see that we have a sequence of inclusions of sets of strongly destabilizing quotients
    $$\gD_{d-1}(\gE(\gamma,K),K)\subset\ldots
\subset \gD_r(\gE(\gamma,K),K).$$ We define successively 
$$ \cW_{d-1}:=\{ \omega_{d-1,F} \ ; \ [F]\in \gD_{d-1}(\gE(\gamma,K),K)\},$$
$$ \cW_{d-2}:=\{ \omega_{d-2,F} \ ; \ [F]\in \gD_{d-2}(\gE(\gamma,K),K)\setminus \gD_{d-1}(\gE(\gamma,K),K)\},$$
$$\ldots,$$
$$ \cW_{r}:=\{ \omega_{r,F} \ ; \ [F]\in \gD_{r}(\gE(\gamma,K),K)\setminus \gD_{r+1}(\gE(\gamma,K),K)\},$$
where 
$$\omega_{s,F}:=\{\alpha_s\in\Sigma_s \ ; \ 
\frac{\deg_{\alpha_s}(F)}{\deg_{\alpha_d}(F)}= \frac{\deg_{\alpha_s}(\gamma)}{\deg_{\alpha_d}(\gamma)}\}
$$
for each $s$ with $d>s\ge r$. One verifies now easily that the sets $\cW_s$, $r\le s\le d-1$ can be used to build a "prismatic" $K$-chamber structure as in the statement of the Proposition.
\end{proof}

\subsection{The general case} A combination of the discussed special cases gives us an existence statement for a chamber structure on $\Amp^{d}(X)\times\ldots\times\Amp^r(X)$  in the general case. 
\begin{prop}\label{prop:chamber-general}
Let $K$ be a compact subset of 
$\Sigma:=\Sigma_{d}\times\ldots\times\Sigma_r$ 
 such that the condition $\BSS(K', \gamma)$ is satisfied. Then for each $s$ with $d>s\ge r$ there exists a finite set $W(\gamma,K)_s$ of rational algebraic hypersurfaces in $\Sigma$ such that 
$$W(\gamma,K):=\bigcup_{s=r}^{d-1} W(\gamma,K)_s$$ builds the wall system of a finite $K$-chamber structure for the variation of semistability on $\Sigma_{d}\times\ldots\times\Sigma_r$.
\end{prop}
\begin{proof}
 As in the case of the fixed top degree, for a compact subset $K$ of $\Sigma$ satisfying $\BSS(K', \gamma)$ we get  a sequence of inclusions of bounded sets of strongly destabilizing quotients
    $$\gD_{d-1}(\gE(\gamma,K),K)\subset\ldots
\subset \gD_r(\gE(\gamma,K),K).$$ We define successively 
$$ W(\gamma,K)_{d-1}:=\{ \omega_{d-1,F} \ ; \ [F]\in \gD_{d-1}(\gE(\gamma,K),K)\},$$
$$ W(\gamma,K)_{d-2}:=\{ \omega_{d-2,F} \ ; \ [F]\in \gD_{d-2}(\gE(\gamma,K),K)\setminus \gD_{d-1}(\gE(\gamma,K),K)\},$$
$$\ldots,$$
$$ W(\gamma,K)_{r}:=\{ \omega_{r,F} \ ; \ [F]\in \gD_{r}(\gE(\gamma,K),K)\setminus \gD_{r+1}(\gE(\gamma,K),K)\},$$
where 
$$\omega_{s,F}:=\{\alpha=(\alpha_d,\ldots,\alpha_r)\in\Sigma \ ; \ 
\frac{\deg_{\alpha_s}(F)}{\deg_{\alpha_d}(F)}= \frac{\deg_{\alpha_s}(\gamma)}{\deg_{\alpha_d}(\gamma)}\}
$$
for each $s$ with $d>s\ge r$. One verifies now easily that the set $$W(\gamma,K):=\bigcup_{s=r}^{d-1} W(\gamma,K)_s$$ builds the wall system of a finite $K$-chamber structure for the variation of semistability on $\Sigma$ as stated.
\end{proof}

\bibliography{bib2021}
\bibliographystyle{amsalpha}


\medskip
\medskip
\begin{center}
\rule{0.4\textwidth}{0.4pt}
\end{center}
\medskip
\medskip
\end{document}